\documentclass[12pt,a4paper,leqno]{amsart}

\usepackage[top=1.1in, bottom=1.3in, left=1.5in, right=1.5in]{geometry}
\usepackage{slashbox}
\usepackage[ansinew]{inputenc}
\usepackage[T1]{fontenc}
\usepackage[english]{babel}
\usepackage{amsthm}
\usepackage{amsfonts}         
\usepackage{amsmath}
\usepackage{amssymb}
\usepackage{breqn}
\usepackage{bm}
\usepackage{url}
\usepackage{esint}
\usepackage{fancyhdr}

\newcommand{\R}{\mathbb{R}}
\newcommand{\C}{\mathbb{C}}

\newcommand{\PP}{\mathbb{P}}

\newcommand\pprec{\prec\mkern-5mu\prec}

\newcommand{\n}{\mathfrak{n}}

\newcommand{\eps}{\varepsilon}

\theoremstyle{plain}
\newtheorem{theorem}{Theorem}
\newtheorem{lemma}[theorem]{Lemma}
\newtheorem{prop}[theorem]{Proposition}

\theoremstyle{remark}
\newtheorem{remark}{Remark}

\numberwithin{equation}{section}
\begin{document}
\title{Exceptional characters and prime numbers in sparse sets}
\author{Jori Merikoski}
\address{Department of Mathematics and Statistics, University of Turku, FI-20014 University of Turku,
Finland}
\email{jori.e.merikoski@utu.fi}

\begin{abstract} 
We develop a lower bound sieve for primes under the (unlikely) assumption of infinitely many exceptional characters. Compared with the illusory sieve due to Friedlander and Iwaniec which produces asymptotic formulas, we show that less arithmetic information is required to prove non-trivial lower bounds. As an application of our method, assuming the existence of infinitely many exceptional characters we show that there are infinitely many primes of the form $a^2+b^8$. 
\end{abstract}

\maketitle

\tableofcontents
\section{Introduction}
Understanding the distribution of prime numbers along polynomial sequences is one of the  basic questions in analytic number theory. For sparse polynomial sequences the problem is solved only in a handful of cases. The most notable are the Friedlander-Iwaniec theorem of primes of the form $a^2+b^4$ \cite{FI} and the result of Heath-Brown of primes of the form $a^3+2b^3$ \cite{hb}, which has been generalized to binary cubic forms by Heath-Brown and Moroz \cite{hbm} and to general incomplete norm forms by Maynard \cite{maynard}. Also, the result of Friedlander and Iwaniec has been extended by Heath-Brown and Li to primes of the form $a^2+p^4$ where $p$ is a prime \cite{hbli}.

Let $\pm D$ be a fundamental discriminant and let $\chi_D(n) = (\frac{D}{n})$ be the associated primitive real character. We say that $\chi_D$ is exceptional if $L(1,\chi_D)$ is very small, say,
\begin{align} \label{introlchibound}
L(1,\chi_D) = \sum_{n=1}^\infty \frac{\chi_D(n)}{n} \leq \log^{-100} D.
\end{align} 
It is conjectured that (for a exponent such as 100) there are at most finitely many exceptional characters, which is closely related to the conjecture that $L$-functions do not have zeros close to $s=1$ (so-called Siegel zeros).  However, assuming that there do exist infinitely many exceptional characters, it is possible to prove very strong results on distribution of prime numbers. For example, Heath-Brown has shown that the twin prime conjecture follows from such an assumption \cite{hbsiegel}, and Drappeau and Maynard have bounded sums of Kloosterman sums along primes \cite{dm}. The potential benefit of such results is that for an unconditional proof we are now allowed to assume the non-existence of exceptional characters, which in turn implies strong regularity in the distribution of primes in arithmetic progressions. Such a bifurcation in the proof has been successfully used to solve problems, for example, in the proof of Linnik's theorem \cite{linnik} and in many results in the theory of $L$-functions.

The state of the art method using exceptional characters is the so-called illusory sieve developed by Friedlander and Iwaniec \cite{fiap,fishort,fiillusory}, which is geared towards counting primes in sparse sets. Assuming the existence of infinitely many exceptional characters (with the exponent 100 in (\ref{introlchibound}) replaced by 200), Friedlander and Iwaniec \cite{fiillusory} proved that there are infinitely many prime numbers of the form $a^2+b^6$. For their method it is required to solve the corresponding ternary divisor problem, that is, show an asymptotic formula for $\sum \tau_3(a^2+b^6)$. This essentially comes down to showing that the sequence has an exponent of distribution $2/3-\varepsilon$. Friedlander and Iwaniec have solved this problem for $a^2+b^6$ in a form that is narrowly sufficient for the illusory sieve \cite{fidivisor}. 

Their method fails for sparser polynomial sequences such as $a^2+b^8$, which has an exponent of distribution $5/8-\varepsilon$. The purpose of this article is to develop a lower bound version of the illusory sieve. That is, instead of aiming for an asymptotic formula for primes of the form $a^2+b^8$, we just want to prove a lower bound of the correct order of magnitude for the number of primes. Morally speaking, we are able to show a non-trivial lower bound for primes in sequences with a level of distribution greater than $(1+\sqrt{e})/(1+2 \sqrt{e}) = 0.61634\dots$ (see Theorem \ref{generaltheorem}), so that the sequence $a^2+b^8$ qualifies.

We will state the general version of our lower bound sieve at the end of this article (Theorem \ref{generaltheorem}). For now we state the result for primes of the form $a^2+b^8$. For any $n \geq 0$ define
\begin{align*}
\kappa_{n} := \int_0^1 \sqrt{1-t^n} dt.
\end{align*}
\begin{theorem} \label{maintheorem}
 If there are infinitely many exceptional primitive characters $\chi$, then there are infinitely many prime numbers of the form $a^2+b^8$. More precisely, if $L(1,\chi_{D}) \leq \log^{-100} D$, then for $\exp(\log^{10} D) < x <\exp(\log^{16} D)$ we have
\begin{align*}
\sum_{\substack{a^2+b^8 \leq x \\ a,b >0}} \Lambda(a^2+b^8) \geq (0.189-o(1)) \cdot \frac{4}{\pi} \kappa_8 x^{5/8}
\end{align*} 
and 
\begin{align*}
\sum_{\substack{a^2+b^8 \leq x \\ a,b >0}} \Lambda(a^2+b^8) \leq (1+o(1)) \cdot \frac{4}{\pi} \kappa_8 x^{5/8}.
\end{align*}
\end{theorem}
\begin{remark}  Note that $\kappa_2 = \pi/4$, so that the coefficient is in fact $\kappa_8/\kappa_2,$ and $\frac{4}{\pi} \kappa_8 x^{5/8}$ is the expected main term. It turns out that the upper bound result is much easier and for this having an exponent of distribution 1/2 is sufficient.
\end{remark}

\subsection{Sketch of the argument} \label{sketchsection}
We present here a non-rigorous sketch of the proof of the lower bound in Theorem \ref{maintheorem}.  Let
\begin{align*}
a_n:= 1_{(n,D)=1} \sum_{\substack{
n=a^2+b^8 \\ (a,b)=1 \\  a,b >0}} 1,
\end{align*}
so that our goal is to estimate $\sum_{n \sim x} a_n \Lambda(n)$.

Let $\chi= \chi_D$. Similarly as in \cite{fiillusory}, we define the Dirichlet convolutions
\begin{align*}
\lambda:=1 \ast \chi \quad  \quad \text{and} \quad \quad \lambda' := \chi \ast \log,
\end{align*}
so that 
\begin{align} \label{lambdaformula}
\lambda \ast \Lambda = (1\ast \chi) \ast (\mu \ast \log) = (\chi \ast \log) \ast (1\ast \mu)=\lambda' .
\end{align}
 Note that $\lambda(n) \geq 0$ and $\lambda'(n) \geq \Lambda(n) \geq 0$ (by  using $\lambda'= \lambda \ast \Lambda$). 

The basic idea in arguments using the exceptional characters is as follows. Since
\begin{align*}
L(1,\chi)^{-1} = \sum_{n} \mu(n) \chi(n) / n = \prod_p \bigg( 1-\frac{\chi(p)}{p}\bigg)
\end{align*} is large, we expect that $\chi(p) = \mu(p)$ for most primes (in a range depending on $D$), so that heuristically we have $\chi \approx \mu$ and $ \lambda' \approx \Lambda$. Hence, we expect that
\begin{align} \label{heuristic1}
\sum_{n \sim x} a_n \Lambda(n) \approx \sum_{n \sim x} a_n \lambda'(n).
\end{align}
Since the modulus of $\chi$ is small, morally $\lambda'(n)$ is of same complexity as the divisor function $\tau(n)$, so that we have replaced the original sum by a much simpler sum.

Making the approximation (\ref{heuristic1}) rigorous is the difficult part of the argument, especially for sparse sequences $a_n$. Friedlander and Iwaniec succeeded in this under the assumption that the exponent of distribution is almost $2/3$, which was sufficient to handle primes in the sequence $a^2+b^6$. 
In our application $a_n$ has the exponent of distribution $5/8-\varepsilon$. This results in an additional error term compared to \cite{fiillusory}, but we are able to show that the contribution from this is smaller (but of the  same order) as the main term.
 
To bound the error term in (\ref{heuristic1}), using $\lambda'= \lambda \ast \Lambda$ we see that
\begin{align*}
\lambda'(n) - \Lambda(n) = \sum_{\substack{n=km \\ m > 1}} \Lambda(k) \lambda(m).
\end{align*}
Let $z=x^\varepsilon$ (in the proof we choose a slightly smaller $z$ for technical reasons). Then
\begin{align*}
\sum_{n \sim x} a_n \Lambda(n) &\geq \sum_{n \sim x} a_n \Lambda(n) 1_{(n,P(z)) =1}  \\
&= \sum_{n \sim x} a_n \lambda'(n) 1_{(n,P(z)) =1} -  \sum_{\substack{km \sim x \\ k,m \geq z}} a_{km} \Lambda(k)\lambda(m)  1_{(km,P(z))=1}  \\
&=: S_1 - S_2.
\end{align*}
Note that by removing the small prime factors we have guaranteed that $m \geq z$ in the second sum, so that we expect $\lambda(m) \approx (1\ast \mu)(m) = 0$ for almost all $m$ in $S_2$.  Thus, we expect that $S_1$ gives us the main term and that $S_2=o(S_1)$.

\begin{remark} The above decomposition has a close resemblance to the recent work of Granville  \cite{granville} using the identity
\[
\Lambda(n) 1_{(n,P(z)) }= 1_{(n,P(z))}\log n - \sum_{\substack{n=\ell m \\ (\ell m ,P(z))=1 \\ \ell, m \geq z}} \Lambda(\ell).
\]
\end{remark}

For the main term $S_1$ we can handle the condition $(n,P(z)) =1 $ by the fundamental lemma of the sieve, so we ignore this detail for the moment. Thus, we have to evaluate
\begin{align*}
\sum_{n \sim x} a_n \lambda'(n) = \sum_{mn \sim x} a_{mn} \chi(m) \log n.
\end{align*}
We have $m \geq x^{1/2}$ or $n \geq x^{1/2}$, so that we are able to compute $S_1$ provided that our sequence $a_n$ has a level of distribution $x^{1/2}$. This is because the modulus of $\chi$ is $x^{o(1)}$, so that $\chi$ is essentially of the same complexity as the constant function $1$. We find that $S_1$ gives the expected main term, so that we need to bound the error term $S_2$.

Similarly as in the argument in \cite{fiillusory}, the range $x^{2/3}$ plays a special role. With this in mind, we define $\gamma = 1/24+ \varepsilon$ so that $2/3-\gamma= 5/8-\varepsilon$ is the exponent of distribution.  We split $S_2$ into three parts depending on the size of $k$
\begin{align*}
S_2 &= \sum_{\substack{km \sim x \\k > x^{1/3+\gamma} \\ m \geq z}} a_{km} \Lambda(k) \lambda(m)  1_{(km,P(z))=1} + \sum_{\substack{km \sim x \\  x^{1/3-2\gamma} < k \leq x^{1/3+\gamma} \\ m \geq z}} a_{km} \Lambda(k) \lambda(m)  1_{(km,P(z))=1} \\
& \hspace{150pt} +\sum_{\substack{km \sim x \\ z \leq k \leq  x^{1/3-2\gamma} \\ m \geq z}} a_{km} \Lambda(k) \lambda(m)  1_{(km,P(z))=1} \\
&=: S_{21} +S_{22} + S_{23}.
\end{align*}

By similar arguments as in \cite{fiillusory}, we are able use the lacunarity of $\lambda(m)$ to bound the terms $S_{21}$ and $S_{23}$ suitably in terms of $L(1,\chi)$, using the fact that the exponent of the distribution is $2/3-\gamma$. That is, for $S_{21}$ we write 
\[
S_{21} \leq (\log x )\sum_{\substack{km \sim x \\k > x^{1/3+\gamma} \\ m \geq z}} a_{km}  \lambda(m)  1_{(m,P(z))=1},
\]
and for $S_{23}$ we drop $1_{(m,P(z))=1}$ by positivity and write
\[
\lambda(m)=\sum_{m=cd} \chi(d),
\] 
where $c$ or $d$  is $> x^{1/3+\gamma}$. In all cases we get a variable $> x^{1/3+\gamma}$, so that these can be evaluated as Type I sums. This gives
\begin{align*}
S_{21} + S_{23} \ll_C x^{5/8}( \log^{-C}x + L(1,\chi) \log ^5 x ),
\end{align*}
which is sufficient by the assumption that $\chi$ is an exceptional character.

The novel part in our argument is the treatment of the middle range 
\begin{align*}
S_{22} = \sum_{\substack{km \sim x \\  x^{1/3-2\gamma} < k \leq x^{1/3+\gamma} \\ m \geq z}} a_{km} \Lambda(k) \lambda(m)  1_{(km,P(z))=1}.
\end{align*}
Note that also in \cite{fiillusory} a narrow range near $x^{2/3}$ has to be discarded, but the argument there requires $\gamma= o (1).$ Thanks to the restriction $(m,P(z))=1$, it turns out that we are able to handle all parts of $S_{22}$ except when $m$ is a prime number. To see this, if $m$ is not a prime, then $m=m_1m_2$ for some $m_1,m_2 \geq z$, and we essentially get (recall that $\lambda(m) \geq 0$)
\begin{align*}
\sum_{\substack{km \sim x \\  x^{1/3-2\gamma} < k \leq x^{1/3+\gamma} \\ m \notin \PP}} a_{km} \Lambda(k) \lambda(m)  1_{(m,P(z))=1} \leq \sum_{\substack{km_1 m_2 \leq x \\  x^{1/3-2\gamma} < k \leq x^{1/3+\gamma} \\ m_1,m_2 \geq z}} a_{km} \Lambda(k) \lambda(m_1) \lambda(m_2) 1_{(m_1 m_2,P(z))=1},
\end{align*}
since $\lambda$ is multiplicative and the part where $(m_1,m_2) > 1$ gives a negligible contribution. For the part  $k m_1 > x^{1/2}$ we use $\lambda(m_1) \leq \tau(m_1) \ll 2^{1/\epsilon}$ and combine variables $\ell = km_1$ to get a bound
\begin{align*}
 \ll \sum_{z \leq m_2 \ll x^{1/2}} \lambda(m_2) \sum_{ \ell \sim x/m_2} a_{\ell m_2},
\end{align*}
which can be bounded suitably in terms of $L(1,\chi)$ by a similar argument as with $S_{21}$. The part $km_1 \leq x^{1/2}$ is handled similarly, using  $\lambda(m_2) \leq \tau(m_2) \ll 2^{1/\epsilon}$ and extracting $L(1,\chi)$ from $\lambda(m_1)$ this time. Thus, the contribution from the composite $m$ is negligible.

Hence, it remains to bound
\begin{align*}
S_{222}:=\sum_{\substack{kp \sim x \\  x^{1/3-2\gamma} < k \leq x^{1/3+\gamma} }} a_{km} \Lambda(k) \lambda(p) = \sum_{\substack{kp \sim x \\  x^{1/3-2\gamma} < k \leq x^{1/3+\gamma} }} a_{km} \Lambda(k) (1+\chi(p)).
\end{align*}
Here we are not able to make use of the lacunarity of $\lambda(p)$. However, since $S_{222}$ counts products of two primes of medium sizes, we immediately see that $S_{222}$ should be smaller than the main term by a factor of $O(\gamma)$, so that at least for small enough $\gamma$ we get a non-trivial lower bound. We use the linear sieve upper bound to the variable $p$ to make this upper bound rigorous and precise, which leads to the constant 0.189 in Theorem \ref{maintheorem}.

The paper is structured as follows. In Section \ref{sievesection} we carry out the sieve argument and the proof of Theorem \ref{maintheorem} assuming a sufficient exponent of distribution for $a_n$ (Propositions \ref{typei1prop} and \ref{typei2prop}). In Section \ref{typeisection} we prove Propositions \ref{typei1prop} and \ref{typei2prop} by generalizing the arguments in \cite{fidivisor}. Lastly, in Section \ref{generalsection} we state a general version of the sieve and explain how the method could be improved assuming further arithmetic information.

\begin{remark}
Our sieve argument is inspired by Harman's sieve method \cite{harman}, although the exact details in this setting turn out to be quite different. The moral of the story is that all sieve arguments should be continuous with respect to the quality of the arithmetic information, which in this case is measured solely by the exponent of distribution. That is, even though we fail to obtain an asymptotic formula after some point (in this case 2/3), we still expect to be able to produce lower and upper bounds of the correct order of magnitude with slightly less arithmetic information.
\end{remark}
\subsection{Notations}
For functions $f$ and $g$ with $g \geq 0$, we write $f \ll g$ or $f= O(g)$ if there is a constant $C$ such that $|f|  \leq C g.$ The notation $f \asymp g$ means $g \ll f \ll g.$ The constant may depend on some parameter, which is indicated in the subscript (e.g. $\ll_{\epsilon}$).
We write $f=o(g)$ if $f/g \to 0$ for large values of the variable. For summation variables we write $n \sim N$ meaning $N<n \leq 2N$. 

For two functions $f$ and $g$ with $g \geq 0$, it is convenient for us to denote  $f(N) \pprec g(N)$ if $f(N) \ll g(N) \log^{O(1)} N$.
For parameters such as $\varepsilon$ we write  $f(N) \pprec_\varepsilon g(N)$ to mean $f(N) \ll_\varepsilon g(N) \log^{O_\varepsilon(1)} N.$ A typical bound we use is $ S(N)=\sum_{n \leq N} \tau_k(n)^K \pprec_{k,K} N$, where $\tau_k$ is the $k$-fold divisor function. We say that an arithmetic function $f$ is divisor bounded if $|f(n)| \,\pprec \tau(n)^K$ for some $K$.

For a statement $E$ we denote by $1_E$ the characteristic function of that statement. For a set $A$ we use $1_A$ to denote the characteristic function of $A.$ 

We let $e(x):= e^{2 \pi i x}$ and $e_q(x):= e(x/q)$ for any integer $q \geq 1$. We denote
\begin{align*}
\lambda:=1 \ast \chi \quad  \quad \text{and} \quad \quad \lambda' := \chi \ast \log.
\end{align*}

\subsection{Acknowledgements}
I am grateful to my supervisor Kaisa Matom\"aki for helpful comments and encouragement.  I also wish to thank Kyle Pratt for comments on an early version of this manuscript. During the work the author was funded by UTUGS Graduate School. Part of the article was also completed while I was working on projects funded by the Academy of Finland (project no. 319180) and the Emil Aaltonen foundation.

\section{The sieve argument} \label{sievesection}
In this section state the arithmetic information (Propositions \ref{typei1prop} and \ref{typei2prop}) and assuming this we give the proof of Theorem \ref{maintheorem} using a sieve argument with exceptional characters. We postpone the proof of Propositions \ref{typei1prop} and \ref{typei2prop} to Section \ref{typeisection}. From here on we let $q$ denote the modulus of the exceptional character $\chi= \chi_{q}$, to avoid conflating it with the level of distribution which we will denote by $D$ (this also agrees with the notations in \cite[Section 14]{fiillusory}). Throughout this section we denote
\begin{align*}
a_n:= 1_{(n,q)=1} \sum_{\substack{
n=a^2+b^8 \\ (a,b)=1 \\  a,b >0}} 1,
\end{align*}
and
\begin{align*}
b_n := 1_{(n,q)=1}\frac{1}{4}   \sum_{\substack{
n=a^2+b^2 \\ (a,b)=1 \\ a,b >0}} b^{-3/4}.
\end{align*}
In $b_n$ we are counting the representations $a^2+b^2$ weighted with the probability that $b$ is a perfect fourth power so that heuristically we expect $\sum_{n\sim x} a_n \Lambda(n) =(1+o(1)) \sum_{n\sim x} b_n \Lambda(n)$. Differing from \cite{fiillusory}, it is convenient for us to write certain parts of the argument as a comparison between $a_n$ and $b_n$. This is inspired by Harman's sieve method \cite{harman}, where the idea is to apply the same combinatorial decompositions to the sums over $a_n$ and $b_n$ and then compare, using positivity to drop certain terms entirely.

We let $g(d)$ denote the multiplicative function defined by
\begin{align} \label{gdefinition}
g(p^k) =1_{p \, \nmid \, q}\frac{\varrho(p^k)}{p^k} \bigg(1+ \frac{1}{p}\bigg)^{-1},
\end{align}
where $\varrho(d)$ denotes the number of solutions to $\nu^2+1 \equiv 0 \,(d)$. Note that for all primes $p$ we have $\varrho(p)=1+\chi_4(p)$.
\subsection{Preliminaries}
We have collected here some standard estimates that will be needed in the sieve argument.
\begin{lemma}\label{gaussprimeslemma}
Let
\begin{align*}
G_q := \prod_{p|q} \bigg(1-\frac{\varrho(p)}{p} \bigg)^{-1}.
\end{align*}
Then
\begin{align*}
\prod_{p \leq z} (1-g(p)) = (1+o(1)) \frac{G_q \zeta(2)}{L(1,\chi_4)} \prod_{p \leq z} (1-1/p)= (1+o(1)) \frac{G_q\zeta(2)}{L(1,\chi_4) e^\gamma \log z}
\end{align*}
and
\begin{align*}
\sum_{n \leq x} \Lambda(n) b_n &= (1+o(1)) \frac{G_q\zeta(2)}{L(1,\chi_4)}\sum_{n \leq x}  b_n = (1+o(1))\frac{4}{\pi} \kappa_8 x^{5/8} \\
&= (1+o(1)) e^{\gamma_1} \log z  \prod_{p \leq z} (1-g(p))\sum_{n \leq x}  b_n,
\end{align*}
where $\gamma_1=0.577\dots$ denotes the Euler-Mascheroni constant.
\end{lemma}
\begin{proof}
The first asymptotic follows from
\begin{align*}
\prod_p \frac{1-g(p)}{1-1/p} = G_q \prod_p( 1- \chi_4(p)/p)(1-p^{-2})^{-1} = \frac{ G_q \zeta(2)}{L(1,\chi_4)} 
\end{align*}
and Merten's theorem. To get the second part we apply Prime number theorem for Gaussian primes $a+ib$, splitting the sum into boxes $(a,b) \in [z_1,z_1+x/\log^{10} x] \times [z_2,z_2+x/\log^{10} x]$ so that $b^{-3/4} =(1+o(1)) z_2^{-3/4}$, noting that the contribution from boxes with $z_1 \leq x/\log^{10} x$ or $z_2 \leq x/\log^{10} x$ is trivially $\ll x^{5/8}/\log x$ (by writing $\Lambda(n) \leq \log x$). Here the condition $(a^2+b^2,q)=1$ implicit in $b_n$ translates into the multiplicative factor $G_q$ (by an expansion using the M\"obius function). For the last asymptotic note that by the change of variables $t=u^{1/4}$
\begin{align*}
\frac{1}{4}\int_0^1  u^{-3/4} \sqrt{1-u^2} dt = \int_0^1 \sqrt{1-t^8} dt = \kappa_8
\end{align*}
and $L(1,\chi_4) = \pi/4$.
\end{proof}

We also require the following basic estimate (see \cite[Lemma 1]{fisieve}, for instance).
\begin{lemma} \label{divisorlemma} For every integer $n$ and every $k \geq 2$ there exists some $d|n$ such that $d \leq n^{1/k}$ and
\begin{align*}
\tau(n) \leq 2^k \tau(d)^k.
\end{align*}
\end{lemma}

To bound the final error term we require the linear sieve upper bound for primes (apply \cite[Theorem 11.12]{odc} with $z=D$ and $s=1$, using $F(1)=2e^\gamma$). 
\begin{lemma}\label{linearlemma} \emph{\textbf{(Linear sieve upper bound for primes).}}
Let $(c_n)_{n \geq 1}$ be a sequence of non-negative real numbers. For some fixed $X_0$ depending only on the sequence $(c_n)_{n \geq 1}$, define $r_d$ for all square-free $d \geq 1$ by
\begin{align*}
\sum_{n \equiv 0 \, (d)} c_n = g_0(d) X_0 + r_d,
\end{align*}
where $g_0(d)$ is a multiplicative function, depending only on the sequence $(a_n)_{n \geq 1}$, satisfying $0 \leq g_0(p) < 1$ for all primes $p.$ Let $D\geq 2$ (the level of distribution). Suppose that there exists a constant $L >0$ that for any $2 \leq w < D$ we have
\begin{align*}
\prod_{w \leq p < D} (1-g_0(p))^{-1} \leq \frac{\log D}{\log w} \bigg(1+\frac{L}{\log w}\bigg).
\end{align*}
Then
\begin{align*}
\sum_{p} c_p  &\leq (1 + O(\log^{-1/6} D)) X_0 2e^{\gamma_1} \prod_{p\leq D} (1-g_0(p)) + \sum_{\substack{d \leq D \\ d \, \, \text{\emph{square free}}}} |r_d|.
\end{align*}
\end{lemma}

The following lemma gives a basic upper bound for smooth numbers (see \cite[Chapter III.5, Theorem 1]{Ten}, for instance). 
\begin{lemma} \label{smoothlemma} For any $2 \leq z \leq y$ we have
\begin{align*}
\sum_{\substack{n \sim y \\ P^+(n) < z}} 1 \, \ll \, y e^{-u/2},
\end{align*}
where $u:= \log y/\log z.$
\end{lemma}
We also need the following simple divisor sum bound.
\begin{lemma} \label{divisorroughlemma}
Let $M \gg 1$ and let $Z= M^{c_1/(\log \log M)^{c_2}}$ for some constants $c_1,c_2 > 0$. Then for any $K>0$
\begin{align*}
\sum_{m \sim M} \tau(m)^K 1_{(m,P(Z))=1} \ll_{c_1,c_2,K} M.
\end{align*}
\end{lemma}
\begin{proof}
For some $L=L(K)$ we have by a standard sieve bound
\begin{align*}
&\sum_{m \sim M} \tau(m)^K 1_{(m,P(Z))=1} \ll \sum_{m \sim M} \tau_L(m) 1_{(m,P(Z))=1} = \sum_{n_1 \cdots n_L \sim M} 1_{(n_1,P(Z))=1} \cdots  1_{(n_L,P(Z))=1} \\
& \ll_{c_1,c_2,K}  \frac{M (\log\log M)^{c_2}}{\log M} \sum_{n_1, \dots, n_{L-1} \ll M}  \frac{1_{(n_1,P(Z))=1} \cdots  1_{(n_{L-1},P(Z))=1}}{n_1 \cdots n_{L-1} } \ll_{c_1,c_2,K} M,
\end{align*}
by computing the sum over $n_j = \max\{n_1,\dots, n_L\}$ first.
\end{proof}

\subsection{Arithmetic information} \label{aisection}
For the sieve argument we need arithmetic information given by the following two propositions, which state that $a_n$ has an exponent of distribution $5/8-\varepsilon$. We will prove these in Section \ref{typeisection}. The first is just a standard sieve axiom on the level of distribution of the sequence $a_n$, and the second is similar but twisted with the quadratic character $\chi.$ For the rest of this section we denote
\[X:= \sum_{n\sim x} b_n.\]
Recall that $X\asymp x^{5/8}$ by Lemma \ref{gaussprimeslemma}.
\begin{prop}\emph{(Type I information).} \label{typei1prop} Let $B>0$ be a large constant and let $\Delta \in [\log^{-B} x, 1]$. Let $D \leq x^{5/8-\eps}$ and $N$ be such that $DN \asymp x$. Let $\alpha(d)$ be divisor bounded coefficients and let $g(d)$ be as in (\ref{gdefinition}). Then for any $C >0$
\begin{align*}
\sum_{d \sim  D} \alpha(d) \sum_{\substack{n \sim x/d \\n \in (N,N(1+\Delta)]}} a_{dn} &= \sum_{d \sim  D} \alpha(d) \sum_{\substack{n \sim x/d \\n \in (N,N(1+\Delta)]}} b_{dn} + O_{B,C}(X \log^{-C} x) 
\end{align*}
and
\begin{align} \nonumber
\sum_{d \leq D}  \alpha(d) \sum_{n \sim x/d} a_{dn} &=\sum_{d \leq D} \alpha(d) \sum_{n \sim x/d} b_{dn} + O_C(X\log^{-C} x) \\ \label{typeimainterm}
&= X \sum_{d \leq D} \alpha(d) g(d) + O_C(X\log^{-C} x).
\end{align}
Furthermore, for $D \leq x^{2/3+\eps}$ we have the last asymptotic
\[
\sum_{d \leq D} \alpha(d) \sum_{n \sim x/d} b_{dn} 
= X \sum_{d \leq D} \alpha(d) g(d) + O_C(X\log^{-C} x)\] 
and for $\Delta=\log^{-B} x$ for any fixed $B>0$ the bound
\[
\sum_{d \leq D} \alpha(d) \sum_{\substack{n \in (N,N(1+\Delta)]}} b_{dn} 
\ll \Delta X \sum_{d \leq D} |\alpha(d)| g(d). 
\]

\end{prop}
\begin{remark}
In our set up the last asymptotic actually holds up to $D \leq x^{1-\eps}$, but we will not need this.
\end{remark}

\begin{prop}\emph{(Type I$_\chi$ information).} \label{typei2prop} Let $B>0$ be a large constant and let $\Delta \in [\log^{-B} x, 1]$. Let $D \leq x^{5/8-\eps}$ and $N$ be such that $DN \asymp x$. Let  $\alpha(d)$ be divisor bounded coefficients. Then for any $C >0$
\begin{align*}
\sum_{d \sim  D} \alpha(d)\sum_{\substack{n \sim x/d \\n \in (N,N(1+\Delta)]}} a_{dn} \chi(n) = \sum_{d \sim  D} \alpha(d) \sum_{\substack{n \sim x/d \\n \in (N,N(1+\Delta)]}} b_{dn} \chi(n) + O_{B,C}(X\log^{-C} x) \ll_{B,C} X\log^{-C} x.
\end{align*}
and
\begin{align*}
\sum_{d \leq D} \alpha(d) \sum_{n \sim x/d} a_{dn} \chi(n) = \sum_{d \leq D} \alpha(d) \sum_{n \sim x/d} b_{dn} \chi(n) + O_C(X\log^{-C} x) \ll_C X\log^{-C} x.
\end{align*}
Furthermore, the bounds for the sums with $b_{dn}$ hold up to $D \leq x^{2/3+\eps}$.
\end{prop}

We will also need the following proposition to bound certain error terms in terms of $L(1,\chi)$. This follows from \cite[Lemmata 3.7 and 3.9]{fiillusory} (as mentioned in \cite[Section 14]{fiillusory}, the  $g(d)$ defined by (\ref{gdefinition}) is easily shown to satisfy the required assumptions). 
\begin{prop} \label{exceptionalprop} \emph{(Exceptional characters).} Let $\lambda:= (1\ast \chi)$. Then for any $x > z \geq q^9$ we have
\begin{align*}
\sum_{n \leq x} \chi(n) g(n) \ll L(1,\chi) 
\end{align*}
and
\begin{align*}
\sum_{z < n \leq x} \lambda(n)g(n) \ll L(1,\chi) \log^2 x. 
\end{align*}
\end{prop}

\subsection{Initial decomposition} \label{initialsection}
Let $\varepsilon>0$ be  small and define the parameter $\gamma:=1/24+\varepsilon$ so that $2/3-\gamma = 5/8-\varepsilon$ is the exponent of distribution of $a_n$.  Using $\lambda'= \lambda \ast \Lambda$ (see (\ref{lambdaformula})) we get
\begin{align*}
\lambda'(n) - \Lambda(n) = \sum_{\substack{n=km \\ m > 1}} \Lambda(k) \lambda(m) .
\end{align*}
Hence, for $z:= x^{1/(\log \log x)^2}$  we have
\begin{align*}
 \sum_{n \sim x} a_n \Lambda(n) &= \sum_{n \sim x} a_n \Lambda(n) 1_{(n,P(z))=1} + O_C(x^{5/8}/\log^C x)  \\
&=\sum_{n \sim x} a_{n} \lambda'(n) 1_{(n,P(z))=1} - \sum_{\substack{km \sim x \\ k,m \geq z}} a_{km} \Lambda(k)\lambda(m)  1_{(m,P(z))=1}  +O_C(x^{5/8}/\log^C x) \\
&=: S_1 - S_2+ O_C(x^{5/8}/\log^C x).
\end{align*}

Similarly as in \cite{fiillusory}, by the lacunarity of $\lambda(m)$ we expect that $S_2 =o(S_1)$, but this is out of reach. We will show that $S_1 = (1+o(1)) \sum_{n \sim x} b_n \Lambda(n)$ and  $S_2 \leq (0.811 + o(1))\cdot \sum_{n \sim x} b_n \Lambda(n)$ which together imply Theorem \ref{maintheorem}.
\begin{remark}
For technical reasons we have chosen $z$ a bit smaller than $x^\varepsilon$ (compare with Section \ref{sketchsection}). This has the benefit that evaluating $S_1$ is a lot easier. On the downside bounding $S_2$ is slightly more difficult and we require Lemma \ref{divisorlemma} for this.
\end{remark}

\subsection{Sum $S_1$}\label{s1section}
Let $D_1:= x^{\varepsilon}$ for some small $\varepsilon >0$. We expand the condition $1_{(n,P(z))=1}$ by using the M\"obius function and split the sum to get
\begin{align*}
S_1 &= \sum_{n \sim x} a_n \lambda'(n) \sum_{d|(n,P(z))} \mu (d)  \\
&=\sum_{n \sim x} a_n \lambda'(n) \sum_{\substack{d|(n,P(z)) \\ d \leq D_1}} \mu (d)+\sum_{n \sim x} a_n \lambda'(n) \sum_{\substack{d|(n,P(z)) \\ d > D_1}}\mu (d) 
=:S'_{1} + R_1.
\end{align*}
To handle the error term $R_1$, note that if $d|P(z)$ and $d > D_1$, then $d$ has a divisor in $[D_1,2zD_1]$. Since $z= x^{1/(\log \log x)^2}$, by Lemma \ref{divisorlemma} (with $k=2$), Proposition \ref{typei1prop}, and Lemma \ref{smoothlemma} we get
\begin{align} \label{smootherror}
R_1 &\ll (\log x) \sum_{\substack{n \sim x  \\ \exists d| (n,P(z)), \, d \in [D_1,2zD_1]}} a_n \tau(n)^2 \\ \nonumber
& \ll (\log x) \sum_{\substack{d \in [D_1,2zD_1] \\ d| P(z)}} \sum_{c \leq (2x)^{1/2}} \tau(cd)^4 \sum_{n \sim x/cd} a_{cdn} \\ \nonumber
&\ll (\log x) x^{5/8} \sum_{\substack{d \in [D_1,2zD_1] \\ d| P(z)}} \sum_{c \leq (2x)^{1/2}} \tau(cd)^4 g(cd) \ll_C x^{5/8} \log^{-C} x
\end{align}
To get the last bound use $\tau(cd)^4g(cd) \leq \tau(cd)^5 /(cd) \leq \tau(c)^5 \tau(d)^5 /(cd)$ and apply Lemma \ref{divisorlemma} to the variable $d$ before using Lemma \ref{smoothlemma}.

For the main term we write
\begin{align*}
S'_1 &= \sum_{\substack{d | P(z) \\ d \leq D_1}} \mu(d) \sum_{\substack{n \sim x \\ n \equiv 0 \,(d)}} a_{n} \lambda'(n)
=\sum_{\substack{d | P(z) \\ d \leq D_1}} \mu(d) \sum_{\substack{mn \sim x \\ mn \equiv 0 \,(d)}} a_{mn} \chi(m) \log n\\
& = \sum_{\substack{d | P(z) \\ d \leq D_1}} \mu(d) \sum_{\substack{mn \sim x \\ mn \equiv 0 \,(d) \\ n > x^{1/2}}} a_{mn} \chi(m) \log n + \sum_{\substack{d | P(z) \\ d \leq D_1}} \mu(d) \sum_{\substack{mn \sim x \\ mn \equiv 0 \,(d) \\ n \leq x^{1/2}}} a_{mn} \chi(m) \log n =: S_{11} + S_{12}.
\end{align*}
We write (denoting $d_1=(m,d)$)
\begin{align*}
S_{11} &= \sum_{\substack{d_1d_2 | P(z) \\ d_1d_2 \leq D_1}} \mu(d_1d_2)
\sum_{\substack{d_1 m \ll x^{1/2} \\ (m,d_2)=1}} \chi(d_1m)
\sum_{\substack{n \sim x/md_1d_2 \\ d_2n > x^{1/2}}} a_{d_1d_2mn}  \log d_2 n 
\end{align*}
We will use Proposition \ref{typei1prop} to evaluate this sum but first we need to remove the cross-condition $d_2n > x^{1/2}$ and the weight $ \log d_2 n$ by using a finer-than-dyadic decomposition to the sums over $d_2$ and $n$. That is, for $\Delta= \log^{-B} x$ for some large $B >0$ we split $S_{11}$ into 
\begin{align*}
\sum_{\substack{i,j \ll \log^{B+1} x \\D_2 =(1+\Delta)^{i} \\ N=(1+\Delta)^{j} \\ D_2 N(1+\Delta)^2 > x^{1/2}}} \sum_{\substack{d_1d_2 | P(z) \\ d_1d_2 \leq D_1 \\ d_2 \in (D_2,D_2(1+\Delta)]}} \mu(d_1d_2)
\sum_{\substack{d_1 m \ll x^{1/2} \\ (m,d_2)=1}} \chi(d_1m)
\sum_{\substack{n \in (N,N(1+\Delta)] \\n \sim x/md_1d_2 \\ d_2n > x^{1/2}}} a_{d_1d_2mn}  \log d_2 n.
\end{align*}
Here we can write
\[
\log d_2 n = \log D_2 N + O(\log^{-B} x),
\]
where the error term will contribute by Lemma \ref{divisorlemma} and Proposition \ref{typei1prop}
\begin{align*}
\ll \log^{-B} x \sum_{n \sim x} \tau_4(n) a_n &\ll \log^{-B}x \sum_{n \sim x} \tau(n)^4 a_n \\
&\ll \log^{-B} x \sum_{d  \ll x^{1/2}} \tau(d)^8  \sum_{n \sim x/d}  a_n \ll_B x^{5/8}\log^{O(1)-B} x
\end{align*}
so that $S_{11}=S_{11}'+O_B(x \log^{O(1)-B}x)$ with
\begin{align*}
S'_{11} := \sum_{\substack{i,j \ll \log^{B+1} x \\D_2 =(1+\Delta)^{i} \\ N=(1+\Delta)^{j} \\ D_2 N(1+\Delta)^2 > x^{1/2}}} \log D_2N \sum_{\substack{d_1d_2 | P(z) \\ d_1d_2 \leq D_1 \\ d_2 \in (D_2,D_2(1+\Delta)]}} \mu(d_1d_2)
\sum_{\substack{d_1 m \ll x^{1/2} \\ (m,d_2)=1}} \chi(d_1m)
\sum_{\substack{n \in (N,N(1+\Delta)] \\n \sim x/md_1d_2 \\ d_2n > x^{1/2}}} a_{d_1d_2mn}.
\end{align*}
The cross-condition $d_2n >x^{1/2}$ holds trivially and may be dropped except in the diagonal part where
\begin{align*}
(1+\Delta)^{-2} x^{1/2} < D_2 N \leq x^{1/2}.
\end{align*}
The contribution from this diagonal part is bounded by using Proposition \ref{typei1prop}
\begin{align*}
& \ll (\log x)\sum_{\substack{i,j \ll \log^{B+1} x \\D_2 =(1+\Delta)^{i} \\ N=(1+\Delta)^{j} \\(1+\Delta)^{-2} x^{1/2} < D_2 N \leq x^{1/2}}}  \sum_{\substack{d_1d_2 | P(z) \\ d_1d_2 \leq D_1 \\ d_2 \in (D_2,D_2(1+\Delta)]}} 
\sum_{\substack{d_1 m \ll x^{1/2} \\ (m,d_2)=1}} 
\sum_{\substack{n \in (N,N(1+\Delta)] \\n \sim x/md_1d_2}} a_{d_1d_2mn}
\\ & \ll_C x^{5/8}\log^{-C} x+ (\log x)\sum_{\substack{i,j \ll \log^{B+1} x \\D_2 =(1+\Delta)^{i} \\ N=(1+\Delta)^{j} \\(1+\Delta)^{-2} x^{1/2} < D_2 N \leq x^{1/2}}}  \sum_{\substack{d_1d_2 | P(z) \\ d_1d_2 \leq D_1 \\ d_2 \in (D_2,D_2(1+\Delta)]}} 
\sum_{\substack{d_1 m \ll x^{1/2} \\ (m,d_2)=1}} 
\sum_{\substack{n \in (N,N(1+\Delta)] \\n \sim x/md_1d_2}} b_{d_1d_2mn}  \\ &\ll_C x^{5/8}\log^{-C} x +(\log^{O(1)} x) x^{5/8} \Delta^{2} \sum_{\substack{i,j \ll \log^{B+1} x \\D_2 =(1+\Delta)^{i} \\ N=(1+\Delta)^{j} \\(1+\Delta)^{-2} x^{1/2} < D_2 N \leq x^{1/2}}}  1  \ll_{B} x^{5/8} \log^{O(1)-B}x
\end{align*}
by choosing $C=B$. Hence, the cross-condition $d_2n > x^{1/2}$ may be dropped and we get $S_{11}=S_{11}''+O_B(x \log^{O(1)-B}x)$ with
\begin{align*}
S''_{11} := \sum_{\substack{i,j \ll \log^{B+1} x \\D_2 =(1+\Delta)^{i} \\ N=(1+\Delta)^{j} \\ D_2 N(1+\Delta)^2 > x^{1/2}}} \log D_2N \sum_{\substack{d_1d_2 | P(z) \\ d_1d_2 \leq D_1 \\ d_2 \in (D_2,D_2(1+\Delta)]}} \mu(d_1d_2)
\sum_{\substack{d_1 m \ll x^{1/2} \\ (m,d_2)=1}} \chi(d_1m)
\sum_{\substack{n \in (N,N(1+\Delta)] \\n \sim x/md_1d_2 }} a_{d_1d_2mn}.
\end{align*}
Applying a similar decomposition to the corresponding sum with $b_{d_1d_2mn}$ and using
Proposition \ref{typei1prop} we get \begin{align*}
S_{11} &=  \sum_{\substack{d_1d_2 | P(z) \\ d_1d_2 \leq D_1}} \mu(d_1d_2)
\sum_{\substack{d_1 m \ll x^{1/2} \\ (m,d_2)=1}} \chi(d_1m) \sum_{\substack{n \sim x/md_1d_2 \\ d_2 n > x^{1/2}}} b_{d_1d_2mn} \log d_2 n + O_C(x^{5/8} \log^{-C} x)\\
&=:  M_{11} + O_C(x^{5/8} \log^{-C} x).
\end{align*}

Similarly, we get by Proposition \ref{typei2prop}  (denoting $d_2=(n,d)$)
\begin{align*}
S_{12} &= \sum_{\substack{d_1d_2 | P(z) \\ d_1d_2 \leq D_1}} \mu(d_1d_2)
\chi(d_1) \sum_{\substack{d_2 n \leq x^{1/2} \\ (n,d_1)=1}} \log d_2 n
\sum_{\substack{m \sim x/nd_1d_2}} a_{d_1d_2mn} \chi(m)  \\
&=
\sum_{\substack{d_1d_2 | P(z) \\ d_1d_2 \leq D_1}} \mu(d_1d_2)
\chi(d_1) \sum_{\substack{d_2 n \leq x^{1/2} \\ (n,d_1)=1}} \log d_2 n
\sum_{\substack{m \sim x/nd_1d_2}} b_{d_1d_2mn} \chi(m)  + O_C(x^{5/8} \log^{-C} x) \\
&=: M_{12} + O_C(x^{5/8} \log^{-C} x)
\end{align*}
That is, in the sums $S_11$ and $S_12$ we have managed to replace $a_n$ by $b_n$. By reversing the steps to recombine we get
\begin{align*}
M_{11} + M_{12} =  \sum_{\substack{n \sim x }} b_{n} \lambda'(n) \sum_{\substack{d | (n,P(z)) \\ d \leq D_1}} \mu(d) =: M_1
\end{align*}
By a similar argument as in (\ref{smootherror}) we can add the part $d > D_1$  back into the sum and we get
\begin{align*}
M_1 &= \sum_{\substack{n \sim x }} b_{n} \lambda'(n)1_{(n,P(z))=1} +   O_C(x^{5/8}/\log^C x)\\
&\geq \sum_{\substack{n \sim x }} b_{n} \Lambda(n)1_{(n,P(z))=1} +  O_C(x^{5/8}/\log^C x)
\end{align*}
by using $\lambda'(n) \geq \Lambda(n)$. 
Thus, by Lemma \ref{gaussprimeslemma} we have $S_1 \geq (1+o(1)) \sum_{n \sim x} b_n \Lambda(n)$, so that for the lower bound result it suffices to show that $S_{2} \leq (0.811 + o(1))\cdot \sum_{n \sim x} b_n \Lambda(n)$. We now proceed to do this, and at the end of this section we will show how to get the upper bound in Theorem \ref{maintheorem}.

\begin{remark} We have used Lemma \ref{smoothlemma} to handle the restriction $(n,P(z))=1$ instead of applying the fundamental lemma of sieve. Thanks to this we were able to use the trivial lower bound $\lambda'(n) \geq \Lambda(n)$ to simplify the evaluation of the main term. 
\end{remark}
\subsection{Sum $S_2$}
Recall that $\gamma=1/24+\eps$ and $2/3-\gamma=5/8-\eps$. We split the sum $S_2$ into three ranges according to the size of $k$
\begin{align*}
S_2 &=\sum_{\substack{km \sim x \\ k,m \geq z}} a_{km} \Lambda(k)\lambda(m)  1_{(m,P(z))=1} \\ 
& = \sum_{\substack{km \sim x \\k > x^{1/3+\gamma} \\ m \geq z}} a_{km} \Lambda(k) \lambda(m)  1_{(m,P(z))=1} + \sum_{\substack{km \sim x \\  x^{1/3-2\gamma} < k \leq x^{1/3+\gamma} \\ m \geq z}} a_{km} \Lambda(k) \lambda(m)  1_{(m,P(z))=1} \\
& \hspace{150pt} +\sum_{\substack{km \sim x \\ z \leq k \leq  x^{1/3-2\gamma} \\ m \geq z}} a_{km} \Lambda(k) \lambda(m)  1_{(m,P(z))=1} \\
&=: S_{21} +S_{22} + S_{23}.
\end{align*}
Using the assumption that $L(1,\chi)$ is small, we will show that the contribution from $ S_{21}$ and $ S_{23}$ is negligible, and that $S_{22} \leq (0.811 + o(1))\cdot \sum_{n \sim x} b_n \Lambda(n)$.
 
\subsubsection{Sum $S_{21}$}
Here we have $k > x^{1/3+\gamma}$, so that by a crude estimate we get
\begin{align*}
S_{21} & = \sum_{\substack{km \sim x \\k \geq x^{1/3+\gamma} \\ m \geq z}} a_{km} \Lambda(k) \lambda(m)  1_{(m,P(z))=1} \\
&\ll (\log x)\sum_{z \leq m \ll x^{2/3-\gamma}} \lambda(m) \sum_{k \sim x/m} a_{k m}:=S_{21}' = M_{21} + R_{21},
\end{align*}
where
\begin{align*}
M_{21} := (\log x) X\sum_{z \leq m \ll x^{2/3-\gamma}} \lambda(m) g(m) \quad \text{and} \quad R_{21}:= S_{21}'-M_{21}.
\end{align*}
By Proposition \ref{typei1prop} we get
\begin{align*}
R_{21} \ll_C x^{5/8} \log^{-C} x,
\end{align*}
and by Proposition \ref{exceptionalprop} we have
\begin{align*}
M_{21} \ll x^{5/8} L(1,\chi) \log^3 x.
\end{align*}
Hence, we have
\begin{align*}
S_{21} \ll_C x^{5/8} L(1,\chi) \log^3 x +  x^{5/8} \log^{-C} x
\end{align*}
 
\subsubsection{Sum $S_{23}$}
Recall that here $m \gg x^{2/3+2\gamma}$. By positivity we may drop the condition $(m,P(z))=1$. Writing
 \begin{align*}
 \lambda(m) = \sum_{cd=m} \chi(d)
 \end{align*}
we split the sum $S_{23}$ into two ranges, $d \leq x^{1/3+ \gamma}$ or $d> x^{1/3+ \gamma}$. We get $S_{23} \leq S_{231}+S_{232}$, where
\begin{align*}
S_{231}&:= \sum_{z \leq k \leq x^{1/3-2\gamma}}\Lambda(k) \sum_{c \ll  x^{2/3-\gamma}/k}  \sum_{\substack{d \sim x/ck  \\ d > x^{1/3+\gamma}}} \chi(d) a_{cdk} \quad \text{and} \\ 
S_{232}&:= \sum_{z \leq k \leq x^{1/3-2\gamma}}\Lambda(k) \sum_{d \leq  x^{1/3+\gamma}} \chi(d) \sum_{\substack{c \sim x/dk }} a_{cdk}.
\end{align*}
By Proposition \ref{typei2prop} we get (after applying a finer-than-dyadic decomposition similarly as with $S_{11}$ to remove cross-conditions)
\begin{align*}
S_{231} \ll_C x^{5/8} \log^{-C} x.
\end{align*}
By Propositions \ref{typei1prop} and \ref{exceptionalprop} we get (since the contribution from $(k,d)>1$ is trivially negligible)
\begin{align*}
S_{232} &= X \sum_{z \leq k \ll x^{1/3-2\gamma}}\Lambda(k) \sum_{d \leq x^{1/3+\gamma}} \chi(d) g(dk) + O_C( x^{5/8} \log^{-C} x) \\
& \ll_C  X \sum_{z \leq k \ll x^{1/3-2\gamma}}\Lambda(k)g(k) \sum_{d \leq x^{1/3+\gamma}} \chi(d) g(d) + x^{5/8} \log^{-C} x \\
&\ll_C x^{5/8} L(1,\chi) \log x + x^{5/8} \log^{-C} x.
\end{align*}
Combining the bounds, we have
\begin{align*}
S_{23} \ll_C x^{5/8} L(1,\chi) \log x + x^{5/8} \log^{-C} x.
\end{align*}

\subsubsection{Sum $S_{22}$}
We have 
\begin{align*}
S_{22} = \sum_{\substack{km \sim x  \\ x^{1/3-2\gamma} <  k \leq x^{1/3+\gamma}} } a_{km} \Lambda(k) \lambda(m)  1_{(m,P(z))=1} 
\end{align*} 
It turns out that we can handle all parts except when $m$ is a prime, so we write
\begin{align*}
S_{22} &= \sum_{\substack{km \sim x  \\  x^{1/3-2\gamma} <  k \leq x^{1/3+\gamma} \\ m \notin \PP}} a_{km} \Lambda(k) \lambda(m)  1_{(m,P(z))=1} + \sum_{\substack{kp \sim x  \\  x^{1/3-2\gamma} <  k \leq x^{1/3+\gamma}  }} a_{kp} \Lambda(k) \lambda(p)  \\
&=: S_{221}+ S_{222}
\end{align*}

In $S_{221}$ we have $m=m_1m_2$ for  $m_1, m_2 \geq z$. Since $(m_1m_2,P(z))=1$, the part where $(m_1,m_2) > 1$ trivially contributes at most $ \ll z^{-1} x^{5/8} \log^{O(1)}x$ which is negligible.
Hence, using $\lambda(m_1m_2) =\lambda(m_1)\lambda(m_2)$ for $(m_1,m_2)=1$ we get
\begin{align*}
S_{221} \leq \sum_{\substack{km_1 m_2 \sim x  \\ x^{1/3-2\gamma} <  k \leq x^{1/3+\gamma} \\ m_1, m_2 \geq z}} a_{km_1 m_2} \Lambda(k) \lambda(m_1)  \lambda(m_2)1_{(m_1 m_2,P(z))=1} + O_C(x^{5/8} \log^{-C} x).
\end{align*}
We split this sum into two parts according to $k m_1 > x^{1/2}$ or $k m_1  \leq  x^{1/2}$. In either case we get $m_j \ll x^{1/2}$ for some $j \in \{1,2\}$.  We combine the variables $\ell=k m_{2-j}$ and use $\lambda(m_{2-j}) \leq \tau(m_{2-j})$ to obtain by Lemma \ref{divisorlemma}
\begin{align*}
S_{221} &\leq (\log x) \sum_{z \leq m \ll x^{1/2}} \lambda(m)  \sum_{\ell \sim x/m} \tau(\ell)1_{(\ell,P(z))=1} a_{\ell m}  + O_C(x^{5/8} \log^{-C} x)\\
& \ll_{K}(\log x) \sum_{z \leq m \ll x^{1/2}} \lambda(m)  \sum_{d \leq x^{1/K}} \tau(d)^K 1_{(d,P(z))=1} \sum_{\ell \sim x/dm} a_{d \ell m}  + O_C(x^{5/8} \log^{-C} x)
\end{align*}
By Proposition \ref{typei1prop} we get (once we choose $K$ large enough so that $1/2+1/K < 2/3-\gamma$)
\begin{align*}
S_{221} \ll_K M_{221}  + O_{C}(x^{5/8} \log^{-C} x),
\end{align*}
 where 
\begin{align*}
M_{221}   = X (\log x) \sum_{\substack{z \leq m \ll x^{1/2} }}  \lambda(m)  \sum_{d \leq x^{1/K}} \tau(d)^K 1_{(d,P(z))=1} g(d)g(m),
\end{align*}
since the contribution from the part the part $(d,m)>1$ is negligible by a trivial bound.
Thus, by Proposition \ref{exceptionalprop} and Lemma \ref{divisorroughlemma} we have
\begin{align*}
M_{221}& \ll_{C} X (\log x) \sum_{d \leq x^{1/K}} \tau(d)^K  g(d) 1_{(d,P(z))=1} \sum_{z \leq m \ll x^{1/2}} \lambda(m) g(m) \\
&\ll_{C} x^{5/8} L(1,\chi) \log^5 x.
 \end{align*} 
 
 Combining the above bounds we get
 \begin{align*}
 S_{221} \ll_C x^{5/8} L(1,\chi) \log^5 x + x^{5/8} \log^{-C} x,
 \end{align*}
 so all that remains is to bound the sum $S_{222}$. The savings here will come from the fact that $k$ is restricted to a fairly narrow range.
\subsection{Bounding the error term $S_{222}$}
We have
\begin{align*}
S_{222} := \sum_{\substack{kp \sim x  \\ x^{1/3-2\gamma} <  k \leq x^{1/3+\gamma} }} a_{kp} \Lambda(k)  (1+\chi(p)) 
\end{align*}
We will apply the linear sieve upper bound to the non-negative sequence
\begin{align*}
c_n := a_{kn}(1+\chi(n))
\end{align*}
with level of distribution $x^{2/3-\gamma}/k$ (note that by exploiting the cancellation from $\chi(n)$ we save a factor of 2 compared to using the trivial bound $\lambda(p) \leq 2$).
For $(d,k)=1$ define $R(d,k)$  by
\begin{align*}
\sum_{\substack{n \sim x/k \\ n \equiv 0 \, (d)}} a_{kn}(1+\chi(n)) = g(d) g(k) X + R(d,k).
\end{align*}
Note that the contribution from sums with $(d,k)>1$ is negligible by trivial estimates. Then by Lemma \ref{linearlemma} with $D_k= x^{2/3-\gamma}/k$ we have
\begin{align*}
S_{222} \leq (1+o(1))M_{222} + R_{222},
\end{align*}
where
\begin{align*}
M_{222} := X \sum_{x^{1/3-2\gamma} <  k \leq x^{1/3+\gamma} } \Lambda(k)g(k) 2e^{\gamma_1} \prod_{p \leq D_k} (1-g(p))
\end{align*}
and
\begin{align*}
R_{222} =  \sum_{\substack{d k \leq x^{2/3-\gamma} \\ (d,k)=1 }} \Lambda(k) |R(d, k)| \, \ll_C x^{5/8}\log^{-C} x
\end{align*}
by Propositions \ref{typei1prop} and \ref{typei2prop}. Applying Lemma \ref{gaussprimeslemma} we get 
\begin{align*}
M_{222} = (2+o(1)) \sum_{n \sim x} b_n \Lambda(n) \sum_{x^{1/3-2\gamma} <  k \leq x^{1/3+\gamma} }\frac{\Lambda(k)g(k)}{\log (x^{2/3-\gamma}/ k)} =: D(\gamma) \sum_{n \sim x} b_n \Lambda(n).
\end{align*}
By the Prime number theorem we have (denoting $k=x^\alpha$)
\begin{align*}
D(\gamma) &\sim  2 \sum_{ x^{1/3-2\gamma} <  k \leq x^{1/3+\gamma} } \frac{\Lambda(k)}{k \log (x^{2/3-\gamma}/k)} \\
& \sim 2 \sum_{ x^{1/3-2\gamma} <  k \leq x^{1/3+\gamma} } \frac{1}{k \log (x^{2/3-\gamma}/k)} \sim 2 \int_{1/3-2\gamma}^{1/3+\gamma} \frac{d\alpha}{2/3-\gamma - \alpha} \\ &\sim 2 \log \frac{1+3 \gamma}{1-6\gamma}.
\end{align*}
We have $D(1/24) < 0.811$. Since $\varepsilon > 0$ can be taken to be arbitrarily small, this implies
\begin{align*}
S_{222} \leq (0.811 + o(1))\cdot \sum_{n \sim x} b_n \Lambda(n),
\end{align*}
completing the proof of Theorem \ref{maintheorem}. \qed
\subsection{Proof of the upper bound result}
We now explain how to get the upper bound result in Theorem \ref{maintheorem}. By Section \ref{s1section} we have by negativity of $S_2$
\begin{align*}
\sum_{n \sim x} a_n \Lambda(n) &\leq S_1 + O_C(x^{5/8}/\log C  x) = \sum_{\substack{n \sim x }} b_{n} \lambda'(n) 1_{(n,P(z))=1}  + O_C(x^{5/8}/\log C  x) \\
&= \sum_{\substack{n \sim x }} b_{n} \Lambda(n) 1_{(n,P(z))=1} + M_2  + O_C(x^{5/8}/\log C  x), 
\end{align*}
where by reversing the initial decomposition on the $b_n$-side (Section \ref{initialsection})
\[
M_2 := \sum_{\substack{km \sim x \\ k,m \geq z}} b_{km} \Lambda(k)\lambda(m)  1_{(m,P(z))=1}
\]
which is the same as  $S_2$ but with $a_n$ replaced by $b_n$. Now $M_2$ can be bounded similarly as $S_2$, except that we decompose with $\gamma=0$ to get $M_{2}=M_{21}+M_{23}$ with
\begin{align*}
M_{21}:=\sum_{\substack{km \sim x \\k > x^{1/3} \\ m \geq z}} b_{km} \Lambda(k) \lambda(m)  1_{(m,P(z))=1} \\
M_{23}:= \sum_{\substack{km \sim x \\   k \leq x^{1/3} \\ m \geq z}} b_{km} \Lambda(k) \lambda(m)  1_{(m,P(z))=1}.
\end{align*}
By similar arguments as above for $S_{21},S_{23}$ we get 
\[
M_{21} + M_{23} \ll_C x^{5/8} L(1,\chi) \log^5 x + x^{5/8}/\log^C x,
\]
since for $b_n$ we have an exponent of distribution $>2/3$ by Propositions \ref{typei1prop} and \ref{typei2prop}. That is, to prove the upper bound we only needed that $a_n$ has an exponent of distribution $1/2+\eps$ instead of $5/8-\eps$.

\section{Type I sums}
\label{typeisection}
In this section we will prove Propositions \ref{typei1prop} and \ref{typei2prop}. The arguments are straightforward generalizations of the arguments in \cite{fidivisor} and \cite[Section 14]{fiillusory}. Since it does not require much additional effort, we give the arguments in this section for the sequences $a^2+b^{2k}$ for any $k \geq 1$, which yields the exponent of distribution $1/2+1/(2k)-\varepsilon$, as claimed in \cite[below Theorem 4]{fidivisor}.

For the arguments in this section it is convenient for us to define $\pprec$ to mean an inequality modulo logarithmic factors, that is, for two functions $f$ and $g$ with $g \geq 0$ we write $f(N) \pprec g(N)$ if $f(N) \ll g(N) \log^{O(1)} N.$ For parameters such as $\varepsilon$ we write  $f(N) \pprec_\varepsilon g(N)$ to mean $f(N) \ll_\varepsilon g(N) \log^{O_\varepsilon(1)} N.$

Proposition \ref{typei1prop} is a consequence of the following proposition, which we will prove in this section.
\begin{prop} \label{typeidyadicprop} Let $M,L, D \gg 1$. Let $k \geq 1$ integer and let $\lambda_\ell$ be a coefficient such that  $|\lambda_\ell|  \leq 1_{\ell=n^k}$. Let $\psi$ denote a fixed $C^\infty$-smooth compactly supported function and denote $\psi_M(x):=\psi(x/M)$. Then for any divisor bounded $\alpha(d)$ and any real number $m_0 \pprec M$ we have 
\begin{align*}
\sum_{d \sim D} \alpha(d) & \bigg( \sum_{\substack{
(\ell,m)=1 \\ \ell \sim L \\ \ell^2+m^2 \equiv 0 \, (d) }} \lambda_\ell \psi_M(m-m_0) - \int \psi_M(t) \,dt \frac{\varrho(d)}{d} \sum_{\substack{
(\ell,d)=1 \\ \ell \sim L}}\lambda_\ell \frac{\varphi(\ell)}{\ell} \bigg)  \\
&\pprec_\varepsilon  M^\varepsilon (L+M)^{1/2} D^{1/2} L^{1/(2k)} .
\end{align*}
\end{prop}
\emph{Proof of Proposition \ref{typei1prop} assuming Proposition \ref{typeidyadicprop}}. For the sequence $b_n$, which counts $n=a^2+b^2$ weighted with $b^{-1+1/k}/k$, we will apply similar arguments as below but with $k=1$, renormalizing the corresponding $\lambda_\ell$ appropriately.
For $a_n$ which counts $n=a^2+b^8$ we write $m=a$ and $\ell=b^4$, so that we are applying the above proposition with $k=4$.  Similarly as with the treatment of the sum $S_{11}$, we use a finer-than-dyadic decomposition to remove the cross-condition $m^2+\ell^2 \sim x$ that is, writing $\Delta = \log^{-B} x$ for some large $B$, we partition the sum into $\ll \Delta^{-2} \log^{2} x$ parts  where $\ell \in [L_0,L_0(1 + \Delta)]$ and $m \in [M_0, M_0(1+\Delta)]$ with $L_0^2+ M_0^2 \sim x$ and $L_0, M_0 \ll \sqrt{x}$. In fact, we need to refine this decomposition so that for $m$ we use a $C^{\infty}$-smooth finer-than-dyadic partition of unity. Then the resulting coefficients for $m$ are  $C^{\infty}$-smooth functions of the form $\psi_M(m-M_0)$, where $M= M_0 \Delta$ is the width of the window around $M_0 \ll \sqrt{x}$. We can now drop the condition $\ell^2+m^2 \sim x$, with an error contribution bounded by $x^{5/8} \log^{-B +O(1)} x$ coming from the edges (where $L_0^2+M_0^2$ is in $[x(1+\Delta)^{-2}, x(1+\Delta)^2]$ or $[2x(1+\Delta)^{-2}, 2x(1+\Delta)^2]$). To see this, note that we have by Proposition \ref{typeidyadicprop}  using $M_0,L_0 \ll x^{1/2}$
\begin{align*}
 \sum_{d \sim D} |\alpha(d)| \sum_{\substack{m  \\ \ell \in [L_0, L_0(1+\Delta)] \\ m^2+\ell \equiv 0 \, (d)}} \lambda_\ell \psi_{\Delta M_0}( m- M_0) 
&\ll_C x^{5/8}\log ^{-C} x + \Delta^{1+1/k} L_0^{1/k} M_0 \sum_{d \sim D} \frac{|\alpha(d)|\varrho(d)}{d} \\
& \ll_C x^{5/8} \log^{-C} x + x^{5/8} \log^{- (1+1/k) B +O(1)} x,
\end{align*}
and that the number of edge cases is $\ll \log^{B+O(1)} x$, so that we save a factor of $\log^{O(1)-B/k} x$, which is sufficient for $B \gg k$.

We can now apply Proposition \ref{typeidyadicprop} in  each of the parts separately. Note that the we have $L, M \ll x^{1/2}$ and $D \ll x^{5/8-\varepsilon}$, so that the error term is bounded by $x^{5/8-\varepsilon/4}$. To remove the condition $(\ell^2+m^2, q)=1$ implicit in Proposition \ref{typei1prop} we may expand using the M\"obius function to get
\begin{align*}
\sum_{\substack{
\ell^2+m^2 \equiv 0 \, (d) \\ (\ell^2+m^2, q)=1}} = \sum_{f | q} \mu(f) \sum_{\substack{
\ell^2+m^2 \equiv 0 \, (df) }}
\end{align*}
since $(d,q)=1$, and apply Proposition \ref{typeidyadicprop} with level $x^{5/8-\varepsilon} q \ll x^{5/8-\varepsilon/2}$.

Denote $\lambda^{(1)}_\ell=1_{\ell=n^k}$ and  $\lambda^{(2)}_\ell =k^{-1}\ell^{-1+1/k}$.  Let $\tilde{g}(d)$ extend $g(d)$ to $(d,q)>1$, that is,
\[
\tilde{g}(p^k) :=\frac{\varrho(p^k)}{p^k} \bigg(1+ \frac{1}{p}\bigg)^{-1}.
\]
We still have to evaluate the main term in Proposition \ref{typeidyadicprop} to get (\ref{typeimainterm}). Recombining the finer-than-dyadic decomposition to a dyadic one for the variable $\ell$, this follows we once show that for $j\in\{1,2\}$
\[
\sum_{d \sim D} \alpha(d)  \int \psi_M(t) \,dt \frac{\varrho(d)}{d} \sum_{\substack{
(\ell,d)=1 \\ \ell \sim L}}\lambda^{(j)}_\ell \frac{\varphi(\ell)}{\ell} =   \sum_{d \sim D} \alpha(d)  \tilde{g}(d) \sum_{\substack{
(\ell,m)=1 \\ \ell \sim L }} \lambda^{(2)}_\ell \psi_M(m-m_0) +O(x^{5/8-\eta}),
\]
which follows easily once we show that
\begin{equation} \label{claim1mainterm}
\begin{split}
& \sum_{d \sim D} \alpha(d)  \int \psi_M(t) \,dt \frac{\varrho(d)}{d} \sum_{\substack{
(\ell,d)=1 \\ \ell \sim L}}\lambda_\ell^{(j)} \frac{\varphi(\ell)}{\ell} \\
&= \sum_{d \sim D} \alpha(d)   \frac{\varrho(d)}{d} \frac{\varphi(d)}{d} \prod_{p|d}(1-p^{-2})^{-1}  \frac{1}{\zeta(2)}  \sum_{\substack{
m \\ \ell \sim L }} \lambda^{(2)}_\ell \psi_M(m-m_0) +O(x^{5/8-\eta}).
\end{split}
\end{equation}
Define
 \begin{align*}
 H_d := \prod_{p \,\nmid d\, } (1-p^{-2}) =   \sum_{(c,d)=1} \frac{\mu(c)}{c^2} = \frac{1}{\zeta(2)} \prod_{p |d } (1-p^{-2})^{-1}
 \end{align*}
and note that
 \begin{align*}
\sum_{\ell \sim L} \lambda^{(1)}_\ell = (1+L^{-\varepsilon_k})\sum_{\ell \sim L} \lambda^{(2)}_\ell
\end{align*}
and
\[
\int \psi_M(t) \,dt = \sum_{m} \psi_M(m-m_0) + O_C(M^{-C}).
\]
Then, since $M \pprec x^{1/2}$, the claim (\ref{claim1mainterm}) follows once we show
\[
\sum_{d \leq D} \frac{\alpha(d) \varrho(d)}{d} \bigg( \sum_{\substack{
(\ell,d)=1 \\ \ell \sim L}}\lambda^{(j)}_\ell \frac{\varphi(\ell)}{\ell} - \frac{\varphi(d)}{d}H_d\sum_{\substack{
 \ell \sim L}}\lambda^{(j)}_\ell \bigg) \pprec 1.
\]
To show this, note also that
 \begin{align*}
 \frac{\varphi(\ell)}{\ell} = \sum_{c| \ell} \frac{\mu(c)}{c}.
 \end{align*}
Then for $\lambda_\ell=1_{\ell=n^k}$ (and similarly for $\lambda_\ell =k^{-1}\ell^{-1+1/k})$
\begin{align*}
\sum_{d \leq D} \frac{\alpha(d) \varrho(d)}{d} &\bigg( \sum_{\substack{
(\ell,d)=1 \\ \ell \sim L}}\lambda_\ell \frac{\varphi(\ell)}{\ell} - \frac{\varphi(d)}{d}H_d\sum_{\substack{
 \ell \sim L}}\lambda_\ell \bigg) \\
 &= \sum_{d \leq D} \frac{\alpha(d) \varrho(d)}{d} \sum_{(c,d)=1} \frac{\mu(c)}{c}\bigg( \sum_{\substack{
(\ell,d)=1 \\ \ell \sim L/c}}\lambda_{c\ell}  - \frac{\varphi(d)}{cd}\sum_{\substack{
 \ell \sim L}}\lambda_\ell \bigg) 
 \\ &= \sum_{d \leq D} \frac{\alpha(d) \rho(d)}{d} \sum_{(c,d)=1} \frac{\mu(c)}{c} \sum_{e|d} \mu(e) \bigg( \sum_{\substack{ \ell \sim L/ce}}\lambda_{ce\ell}  - \frac{1}{ce}\sum_{\substack{
 \ell \sim L}}\lambda_\ell \bigg)  \\
 &=  \sum_{d \leq D} \frac{\alpha(d) \rho(d)}{d} \sum_{(c,d)=1} \frac{\mu(c)}{c} \sum_{e|d} \mu(e) \bigg( \sum_{\substack{ n \sim L^{1/k}/ce}}1  - \frac{1}{ce}\sum_{\substack{n \sim L^{1/k}}}1 \bigg)  \\
& \ll \sum_{d \leq D} \frac{|\alpha(d) | \rho(d) }{d} \sum_{e|d}  \bigg(\sum_{c \ll L^{1/k}/e}  \frac{1}{c}  + \frac{L^{1/k}}{e} \sum_{c \gg L^{1/k}/e} \frac{1}{c^2}\bigg)\pprec 1
\end{align*}
by writing $\ell=(nce)^k$ since $ce$ is square free. 
\qed

Proposition \ref{typei2prop} follows by a similar argument from the following (recall that $a_n$ and $b_n$ are supported on $(n,q)=1$). 
\begin{prop} \label{typei2dyadicprop}
Let $M,L, D \gg 1$. Let $k \geq 1$ integer and let  $\lambda_\ell$ be  a coefficient such that  $|\lambda_\ell|  \leq1_{\ell=n^k}$. Let $\psi$ denote a fixed $C^\infty$-smooth compactly supported function and denote $\psi_M(x):=\psi(x/M)$. Let $\chi$ denote a primitive quadratic Dirichlet character associated to a fundamental discriminant $ \pm q$ with $q > 1$. Then for any divisor bounded $\alpha(d)$ and any real number $m_0 \pprec M$ we have 
\begin{align*}
\sum_{d \sim D} \alpha(d)   & \sum_{\substack{
(\ell,m)=1 \\ \ell \sim L \\ \ell^2+m^2 \equiv 0 \, (d) }} \lambda_\ell  \psi_M(m-m_0)\chi(\ell^2+m^2)  \\ 
&\pprec_\varepsilon q^2 M^\varepsilon    (L+M)^{1/2} D^{1/2} L^{1/(2k)}  + q^{-\eta} M L^{1/k}.
\end{align*}
\end{prop}

For the proof of Propositions \ref{typeidyadicprop} and \ref{typei2dyadicprop} we need the following large sieve inequality (see \cite[Lemma 14.4]{fiillusory} for the proof).
\begin{lemma} \label{largesieve} Let $q \geq 1$. Then for any complex numbers $\alpha_n$ we have
\begin{align*}
\sum_{\substack{d \sim D \\ (d,q)=1}} \sum_{\nu^2+1 \equiv 0 \, (d)} \bigg| \sum_{n \leq N} \alpha_n e_d(\nu n \bar{q})\bigg| \ll (Dq + N) \sum_{n \leq N} |\alpha_n|^2,
\end{align*}
where $q\bar{q}  \equiv 1 \, (d)$.
\end{lemma}
We also require the Poisson summation formula.
\begin{lemma}\emph{\textbf{(Truncated Poisson summation formula).}} \label{poisson}
Let $\psi:\R\to  \C$ be a fixed $C^\infty$-smooth compactly supported function with $\|\psi\|_{1} \leq 1$ and let $M \gg 1$. Fix a real number $m_0$. Let $d \geq 1$ be an integer. Then for any $\varepsilon > 0$ we have uniformly in $m_0$
\begin{align*}
\sum_{m \equiv a \, (d)} \psi_M(m-m_0 )  = \int \sum_{0 \leq |h| \leq M^\varepsilon d/M} \psi_M(t d-m_0)e(ht)e_d (-ah) dt + O_{C,\varepsilon}(M^{-C}).
\end{align*}
\end{lemma}
\begin{proof}
Applying the Poisson summation formula we get
\begin{align*}
\sum_{m \equiv a \, (d)} \psi_M(m-m_0 ) &= \sum_{n} \psi_M(nd+a-m_0) = \sum_{h} \int \psi_M (td+a-m_0)e(ht) d t \\
&=\sum_{h} \int \psi_M(td-m_0)e(ht) e_d(-ha)d u .
\end{align*}
by the change of variables $t \mapsto t -a/d$. For $|h| > M^{\varepsilon}d/M$ we can iterate integration by parts to show that the contribution from this part is $\ll_{C,\varepsilon} M^{-C}.$ 
\end{proof}
We also need the following Weil bound for character sums.
\begin{lemma}
\label{charactersumlemma}
Let $q \geq 1$ and let $\chi$ be a primitive quadratic character of modulus $q$. Let $a, b \in \mathbb{Z}$ and $(a,q)=1$. Then
\[
\sum_{m \,(q)} \chi(am^2+b)) \ll_\eps (b, q)^{1/2} q^{1/2+\eps}.
\]
\end{lemma}

\subsection{Proof of Propositions \ref{typeidyadicprop} and \ref{typei2dyadicprop}}
We first note that there is a gap in the proof given in \cite[Section 14]{fiillusory}, namely, the argument around their application of Poisson summation works only if the sum is restricted to $(\ell,q)=1$. To fix this we must first bound the contribution $\ell = n^k$ which have a large factor whose prime factors divide $q.$ Let $q_0=q_0(n)=q_0(\ell)$ denote the smallest factor of $n$ such that $(n/q_0, q) = 1$. The parts of the sums in Proposition \ref{typei2dyadicprop} where $q_0 > q^\eta$ can be bounded trivially. To see this,  note that by the divisor boundedness $\alpha(d)$ and Lemma \ref{divisorlemma} we have
\begin{align*}
&\sum_{d \sim D} \alpha(d)    \sum_{\substack{
(\ell,m)=1 \\ \ell \sim L \\ \ell^2+m^2 \equiv 0 \, (d) \\ q_0 > q^\eta }} \lambda_\ell  \psi_M(m-m_0)\chi(\ell^2+m^2) \ll\sum_{\substack{
m \asymp m_0  \\ n \sim L^{1/k} \\ q_0 > q^\eta }} \tau(m^2+n^{2k})^{O(1)}  \\
& \ll \sum_{d \ll m_0^{1/2}} \tau(d)^{O(1)} \sum_{\substack{ n \sim L^{1/k}  \\ q_0 > q^\eta}} \sum_{\substack{m \asymp m_0 \\ m^2 \equiv - n^{2k} \, (d) }} 1 \pprec M \sum_{d \ll m_0^{1/2}} \frac{\tau(d)^{O(1)}}{d} \sum_{\substack{ n \sim L^{1/k}  \\ q_0 > q^\eta}} 1 \pprec M \sum_{\substack{ n \sim L^{1/k}  \\ q_0 > q^\eta}} 1 
\end{align*}
and
\begin{align*}
\sum_{\substack{ n \sim L^{1/k} \\ q_0 > q^\eta}} 1  \leq \sum_{\substack{q_0 > q^\eta \\ p| q_0 \Rightarrow p| q}} \sum_{n \sim L^{1/k}/q_0} \ll L^{1/k}  \sum_{\substack{q_0 > q^\eta \\ p| q_0 \Rightarrow p| q}} q_0^{-1} \leq q^{-\eta/2} L^{1/k} \prod_{p| q} ( 1- p^{-1/2})^{-1} \ll  q^{-\eta/4} L^{1/k}.
\end{align*}
Hence, we may assume that $\lambda_\ell$ is supported on $q_0(\ell) < q^{\eta}$ for some small $\eta>0$.

Note that since $d| \ell^2+m^2$, we may add the condition $(d,q)=1$ since otherwise $\chi(\ell^2+m^2)=0$. Expanding the condition $(\ell,m)=1$ using the M\"obius function, we get 
\begin{align*}
\sum_{\substack{d \sim D \\ (d,q)=1}} \alpha(d)   \sum_{\substack{
(\ell,m)=1 \\ \ell \sim L \\ \ell^2+m^2 \equiv 0 \, (d)}}& \lambda_\ell  \psi_M(m-m_0) \chi(\ell^2+m^2) \\
&= \sum_{\substack{b \ll LM\\ (b,q)=1}} \mu(b) \sum_{\substack{d \sim D \\ (d,q)=1}} \alpha(d)   \sum_{\substack{
 \ell \sim L /b \\ (\ell,d)=1}} \lambda_{b\ell}\sum_{\substack{m \\ b^2(\ell^2+m^2) \equiv 0 \, (d) }} \psi_{M/b}(m-m_0/b) \chi(\ell^2+m^2) \\
\end{align*}
Writing $b_1=(d,b)$ and $b_2=b/b_1$ we get (absorbing $(d,b_2)=1$ into the coefficient $\alpha(d)$ and redefining $\alpha(d)$ as $\alpha(b_1d)$)
\[
\sum_{\substack{b_1b_2 \ll LM\\ (b_1b_2,q)=1}} \mu(b_1b_2) \sum_{\substack{d \sim D/b_1 \\ (d,q)=1}} \alpha(d)   \sum_{\substack{
 \ell \sim L /b \\ (\ell,d)=1}} \lambda_{b\ell}\sum_{\substack{m \\ \ell^2+m^2 \equiv 0 \, (d) }} \psi_{M/b}(m-m_0/b) \chi(\ell^2+m^2).
\]
Let $q_\ell:= q_0^k$ so that $(q,\ell/q_\ell) =1$. Defining $\nu\,(d)$ and $\beta \,(q)$ so that $m \equiv \nu \ell \, (d)$ and $m \equiv \beta (\ell/q_\ell) \, (q)$ we get by the Chinese remainder theorem
\begin{align*}
m \equiv \nu \ell q\bar{q} + \beta (\ell/q_\ell) d \bar{d} \,\, (dq),
\end{align*} 
where the inverses $\bar{q}$ and $\bar{d}$ are computed modulo $d$ and $q$, respectively.
Using Lemma \ref{poisson} we get for $H:= M^{\varepsilon} b_2 Dq/M$
\begin{align*}
&\sum_{\substack{m \\ \ell^2+m^2 \equiv 0 \, (d) }} \psi_{M/b}(m-m_0/b) \chi(\ell^2+m^2) \\
& = \sum_{ \substack{ \nu \, (d)  \\ \nu^2+1 \equiv 0 \,(d)}} \sum_{\beta \, (q)}\chi(\beta^2+q_\ell^2)\sum_{\substack{m \\ m \equiv \nu \ell q\bar{q} + \beta (\ell/q_\ell) d \bar{d} \, (dq) }} \psi_{M/b}(m-m_0/b) \\
&=\sum_{ \substack{ \nu \, (d)  \\ \nu^2+1 \equiv 0 \,(d)}} \sum_{\beta \, (q)} \chi(\beta^2+q_\ell^2) \int \sum_{0 \leq |h| \leq H} \psi_{M/b}(t dq -m_0/b)e(ht)e_{d} (-\nu h \ell \bar{q}) e_{q} (-\beta h (\ell/q_\ell) \bar{d})dt\\
& \hspace{360pt} + O_{C,\varepsilon}(M^{-C}).
\end{align*}
Making the change of variables and $\beta \mapsto \beta d$ this becomes
\begin{align*}
\int \sum_{0 \leq |h| \leq H} \bigg(\sum_{\substack{ \beta \, (q)}} \chi(\beta^2 d^2+q_\ell^2) e_q(-\beta h (\ell/q_\ell))\bigg) \sum_{\substack{ \nu^2+1 \equiv 0 \,(d)}}    \psi_{M}(t b d q -m_0)e(ht)e_{d} (-\nu h \ell \bar{q}) dt.
\end{align*}

From $h=0$ we get a total contribution
\begin{align*}
\sum_{\substack{b_1b_2 \ll LM\\ (b_1b_2,q)=1}} \mu(b_1b_2)\sum_{\substack{d \sim D/b_1 \\ (d,q)=1}} \alpha(d)  \varrho(d) \sum_{\substack{\ell \sim L/b \\ (\ell,d)=1}} \lambda_{b\ell} \frac{M}{b_2dq}\int \psi(t) dt \sum_{\substack{ \beta \, (q)}} \chi(\beta^2 d^2+q_\ell^2) 
\pprec  q^{-1/4} M L^{1/k}
\end{align*}
by using the bound (Lemma \ref{charactersumlemma})
\begin{align*}
\sum_{\substack{ \beta \, (q)}} \chi(\beta^2 d^2+q_\ell^2) \ll_\varepsilon (q,q_\ell^2)^{1/2} q^{1/2+\varepsilon}
\end{align*}
 and the fact that $q_\ell = q_0^k \ll q^{\eta k}$ for some small $\eta$.

For $h \neq 0$ we can by symmetry restrict to $h<0$. We first want to remove the cross-condition $\chi(\beta^2 d^2+q_\ell^2)$ between the variables $d$ and $\ell$. To do this we fix the value of $q_\ell$ modulo $q$ and split $\ell$ into congruence classes $q_\ell \equiv \gamma \, (q).$ Hence, we get for some $|c_{h,\ell}(t,q,\beta,\gamma)| \leq 1$ and $|c_{h,\ell}(t,q)| \leq 1$ that the total contribution from $h \neq 0$ is
\begin{align*}
&  \sum_{\gamma \, (q)} \sum_{\substack{b_1b_2 \ll LM\\ (b_1b_2,q)=1}} \mu(b_1b_2)\int \sum_{\beta \, (q)} \sum_{\substack{d \sim D/b_1 \\ (d,q)=1}} \alpha(d) \chi(\beta^2 d^2 +\gamma^2)   \sum_{\nu^2+1 \equiv 0 \, (d)} \\
& \hspace{150pt}\sum_{\substack{
 \ell \sim L /b \\ (\ell,d)=1 \\ q_\ell \equiv \gamma \, (q)}} \lambda_{b\ell}  \sum_{1 \leq h \leq H} c_{h,\ell}(t,q,\beta,\gamma)  e_d(\nu h \ell \bar{q})\psi_{M}(t b d q -m_0) dt \\ 
&  \pprec q^2 \sum_{b_1b_2 \ll LM} \int \sum_{\substack{d \sim D/b_1 \\ (d,q)=1}} |\alpha(d)| \sum_{\nu^2+1 \equiv 0 \, (d)} \bigg|  \sum_{\substack{
 \ell \sim L /b \\ (\ell,d)=1}} \lambda_{b\ell}  \sum_{1 \leq h \leq H} c_{h,\ell}(t,q)  e_d(\nu h \ell \bar{q})\psi_{M}(t bd q -m_0) \bigg|  dt.
\end{align*}
Note that $\psi_{M}(t bd q -m_0)$ vanishes outside $|tbdq - m_0| \, \ll M$. Hence, by $d \sim D/b_1$ and $m_0 \pprec M$ the integral over $t$ is supported on a fixed set $T(b_1,b_2)$ with measure bounded by $\pprec M/b_2qD$ so that by taking the maximal $t$ the last expression is bounded by
\begin{align*}
 \pprec  q \sum_{b_1b_2 \ll LM}\frac{M}{b_2D} \sum_{\substack{d \sim D/b_1 \\ (d,q)=1}}|\alpha(d)| \sum_{\nu^2+1 \equiv 0 \, (d)} \bigg| \sum_{\substack{
 \ell \sim L /b \\ (\ell,d)=1}} \lambda_{b\ell} \sum_{1 \leq h \leq H} c_{h,\ell}  e_d(\nu h \ell \bar{q}) \bigg|
\end{align*}
for some coefficients $c_{h,\ell}=c_{h,\ell}(b_1,b_2,q,m_0)$ independent of $d$ with $|c_{h,\ell}| \leq 1$. Expanding the condition $(\ell,d)=1$ this is bounded by
\begin{align} \label{sumbeforelargesieve}
\frac{qM}{D}\sum_{b_1b_2 \ll LM} \frac{1}{b_2} \sum_{c \ll DL}  \sum_{\substack{d \sim D/b_1c \\ (d,q)=1}}|\alpha(cd)| \sum_{\nu^2+1 \equiv 0 \, (d)} \bigg| \sum_{\substack{
 \ell \sim L /bc }} \lambda_{bc\ell}  \sum_{1 \leq h \leq H} c_{h,c\ell}  e_d(\nu h c \ell  \bar{q}) \bigg|.
\end{align}
By Cauchy-Schwarz and Lemma \ref{largesieve} the sum over $d$ is bounded by (denoting  $H_1 := H/b_2 = M^{\varepsilon} Dq/M$)
\begin{align*}
&\pprec \frac{D^{1/2}}{(b_1c)^{1/2}} \bigg(  \sum_{\substack{d \sim D/b_1c \\ (d,q)=1}} \sum_{\nu^2+1 \equiv 0 \, (d)} \bigg| \sum_{\substack{
 \ell \sim L /bc }} \lambda_{bc\ell}  \sum_{1 \leq h \leq H} c_{h,c\ell}  e_d(\nu h c\ell \bar{q}) \bigg|^2 \bigg)^{1/2}  \\  &\ll \frac{D^{1/2}}{(b_1c)^{1/2}} ( Dq/b_1c + H L/b)^{1/2} \bigg( \sum_{1 \leq j \ll H_1 L/c} \bigg|\sum_{\substack{j = \ell h \\ \ell \sim L/bc}} \lambda_{bc\ell} \bigg|^2 \bigg)^{1/2}  \\
&\ll \frac{1}{bc^{1/2}}( Dq + (DH_1 L)^{1/2} )  \bigg( \sum_{1 \leq j \ll H_1 L/c} \bigg|\sum_{\substack{j = \ell h \\ \ell \sim L/bc}} \lambda_{bc\ell} \bigg|^2 \bigg)^{1/2}.
\end{align*}
 By Cauchy-Schwarz we get (writing $m=bcj=bj'$ and $B:= LM$ so that $1/b =j'/m \ll H_1 L/m$)
\begin{align*}
\sum_{b \ll LM} \frac{\tau(b)}{b} &\sum_{c \ll DL} \frac{1}{c^{1/2}} \bigg( \sum_{1 \leq j \ll H_1 L/c}\bigg|\sum_{\substack{j = \ell h \\ \ell \sim L/bc}} \lambda_{bc\ell}  \bigg|^2 \bigg)^{1/2} \pprec \bigg( \sum_{\substack{j' \ll H_1 L \\ b \ll B}} \frac{1}{b} \tau(j')  \bigg|\sum_{ \substack{j' = \ell h \\ \ell \sim L/b}} \lambda_{b\ell} \bigg|^2 \bigg)^{1/2} \\
\\
&\ll   \bigg( \sum_{m \ll H_1 L B} \frac{H_1 L}{m}\tau(m)^2 \bigg|\sum_{ \substack{m = \ell h \\ \ell \sim L}} \lambda_{\ell} \bigg|^2 \bigg)^{1/2}  
\leq \bigg( H_1 L \sum_{ n^k \sim L} \sum_{h \ll H_1 B} \frac{\tau(h n^k)^4}{h n^k} \bigg)^{1/2}
\\
& \leq  \bigg( H_1 L \sum_{ n^k \sim L} \sum_{h \ll H_1 B} \frac{\tau(h)^4\tau( n)^{4k} }{h n^k} \bigg)^{1/2} \pprec  H_1^{1/2} L^{1/2k}.
\end{align*}
Hence, the final bound for (\ref{sumbeforelargesieve})
is
\begin{align*}
& \pprec \frac{qM}{D} (Dq + (DH_1L)^{1/2})  H_1^{1/2} L^{1/(2k)} \\
& = q M^{\varepsilon}  (M q^{1/2}H_1^{1/2} L^{1/(2k)} + M H_1 L^{1/2+1/(2k)} D^{-1/2} )  \\
& = M^{\varepsilon} q^2 ( D^{1/2}M^{1/2}  L^{1/(2k)} + D^{1/2} L^{1/2+1/(2k)}  )
\end{align*} 
 by using $H_1= M^{\varepsilon} Dq/M$. \qed
 
 \section{A General version of the sieve}  \label{generalsection}
From our argument in Section \ref{sievesection} we can infer the following general result. We have not made an effort to minimize the assumptions or optimize the powers of logarithms. 
 \begin{theorem} \label{generaltheorem}
Let $x$ be large and let $\chi_D$ be a real primitive character associated to a fundamental discriminant $D=x^{o(1)}$ with $D \gg_C \log^C x$. Let $a_n$ and $b_n$ be non-negative sequences supported on $(n,D)=1$, and let $g(d)$ be the associated multiplicative function. Suppose that $g(d) \ll \tau(d)^{O(1)}/d.$ Assume that $g$ satisfies the assumptions of Lemma \ref{linearlemma} and assume that Proposition \ref{exceptionalprop} holds. Suppose that for any $z > x^\varepsilon$ we have
 \begin{align*}
 \sum_{n \sim x} b_n \Lambda(n) = (1+o(1)) \frac{1}{e^{\gamma_1} \log z} \prod_{p \leq z} (1-g(p)) \sum_{n \sim x} b_n
 \end{align*}
 and
 \begin{align*}
 \sum_{k \sim z} \Lambda(k) g(k) = (1+o(1)) \sum_{k \sim z} \frac{\Lambda(k)}{k}.
 \end{align*}
 Suppose also that for some $\epsilon>0$ we have the crude bounds
 \begin{align*}
 \sum_{n \sim x} a_n \Lambda(n) 1_{(n,P(x^\epsilon))>1} ,  \quad \sum_{n \sim x} b_n \Lambda(n) 1_{(n,P(x^\epsilon))>1} = o (\sum_{\substack{n \sim x}} \Lambda(n) b_n ).
\end{align*}  
 Suppose that the exponent of distribution is at least $\alpha=2/3-\gamma$ for some $\gamma < 1/6$ (in the sense of Propositions \ref{typei1prop} and \ref{typei2prop}). 
Then
 \begin{align*}
\sum_{\substack{n \sim x}} \Lambda(n) a_n \geq \bigg(1- 2 \log \frac{1 + 3\gamma}{1-6\gamma}-O(L(1,\chi_D)\log^5 x)-o(1) \bigg)\sum_{\substack{n \sim x}} \Lambda(n) b_n.
\end{align*} 
Assuming that the exponent of distribution is at least $1/2+\eps$ we have
 \begin{align*}
\sum_{\substack{n \sim x}} \Lambda(n) a_n \leq (1+O(L(1,\chi_D)\log^5 x)+o(1))\sum_{\substack{n \sim x}} \Lambda(n) b_n.
\end{align*} 
In particular, if $L(1,\chi_{D}) \leq \log^{-100} D$ and $\exp(\log^{10} D) < x <\exp(\log^{16} D)$, then the lower bound is non-trivial as soon as the exponent of distribution satisfies 
\begin{align*}
\alpha > \frac{1+\sqrt{e}}{1+2 \sqrt{e}}= 0.61634\dots
\end{align*}
 \end{theorem}
 \begin{remark}With much more effort it is possible to get the same result as above with $L(1,\chi)\log x$  in place of $L(1,\chi)\log^5 x$, so that one only needs $L(1,\chi_D)=o(1/\log D)$.
 \end{remark}
 \begin{remark} Unfortunately the above theorem just misses out the next case  $a^2+b^{10}$, which has an exponent of distribution $3/5-\varepsilon.$ Similarly as with the linear sieve, further improvements are possible if we make use of well-factorability of the weights \cite[Chapter 12.7]{odc}. For example, the upper bound for the sum $S_{222}$ can be improved if we are able to handle certain Type I/II sums (that is, Type I sums where the modulus is $kd$ with $d$ well-factorable). Note also that in $S_{21}$ and $S_{23}$ the weight factorizes and furthermore there is some smoothness available in the weight. Hence, assuming suitable arithmetic information (of Type I/II or Type I$_2$) we could handle some parts near the edges of $S_{22}$ by a similar argument as for the sums $S_{21}$ or $S_{23}$. Unfortunately we do not know how to carry this out for the sequence $a^2+b^{10}$, but possibly sums of Kloosterman sums methods might be able to handle these sums. It is also unclear if the handling of the sum $S_{222}$ is optimal but we have not found a way to improve this.
 \end{remark}
 
\begin{remark}
The ideas in this paper can be used also to the problem of primes in short intervals, to improve the result of Friedlander and Iwaniec \cite{fishort} which gives primes in intervals of length $x^{39/79} < x^{1/2}$ under the assumption of exceptional characters. The sieve argument is slightly different here since for this problem we can also utilize  the available Type I/II  and Type I$_2$ information furnished by the exponential sum estimates used for the problem of largest prime factor on short intervals \cite{bhshort,
fwshort,lwshort}. The details will appear elsewhere.
\end{remark} 
 
 \bibliography{charactersievebib}
\bibliographystyle{abbrv}
  \end{document}